\documentclass{article}
\usepackage{amsmath, amsfonts, amsthm}
\theoremstyle{plain}\newtheorem{lemma}{Lemma}
\theoremstyle{plain}
\theoremstyle{plain}\newtheorem{thm}{Theorem}
\theoremstyle{plain}\newtheorem{prop}{Proposition}
\title{Zeros of $2\times2$ Matrix Polynomials}
\author{Marla Slusky\footnote{This paper represents results obtained 
at a Research Experience for Undergraduates program at Rutgers University. 
The author thanks her mentor, Robert Wilson, for his ideas and input.}}

\begin{document}
\newcommand{\N}{\mathbb{N}}
\newcommand{\C}{\mathbb{C}}
\maketitle
\section*{Abstract}
Consider the $n$th degree polynomial equation, 
$X^n+A_{n-1}X^{n-1}+\cdots+A_1X+A_0=0$ over the ring of 
$2\times2$ complex matrices. 
If this equation has more than $\binom{2n}{2}$ solutions, then it 
has infinitely many solutions. We show here that for any 
$n,m \in\N$ such that $m\leq\binom{2n}{2}$, there 
exists an $n$th degree polynomial equation with exactly $m$ 
solutions.
\section{Introduction}
Let $R$ be a ring and let
\begin{equation*}
f(X)=X^n+A_{n-1}X^{n-1}+\dots+A_1X+A_0
\end{equation*}
be a polynomial in the indeterminate $X$ with coefficients in $R$
(with powers of $X$ on the right side of the coefficients).
It is well known that if $R$ is a field, the number of solutions in 
$R$ of $f(X)=0$ is $\leq n$. If $R$ is a division ring, then (\cite{BW})
the number of solutions in $R$ of $f(X)=0$ is either $\leq n$ or 
infinite. In case $R$ is the ring of $k\times k$ matrices over $\C$, 
a number of authors (e.g., \cite{FS}, \cite{G}, \cite{LR}, \cite{GLR})
have studied solutions of $f(X)=0$. In particular, Fuchs
and Schwartz have shown that in the generic case, $f(X)=0$ has 
$\binom{kn}{k}$ diagonalizable solutions.

In this paper, we consider the case where $R$ is the ring of $2\times 2$
matrices over the complex numbers, $M_2(\C)$. It is true
(Proposition \ref{inf}) that if the equation $f(X)=0$ has more than
$\binom{2n}{2}$ solutions then it has infinitely many. The main result
of this paper is that every number of solutions $\leq \binom{2n}{2}$ can
arise. That is, 
\begin{thm}\label{main}
Given $m,n\in\N$, $m\leq \binom{2n}{2}$, there exists an $n$th degree equation
over $2\times2$ complex matrices that has exactly $m$ solutions.
\end{thm}

The paper is organized as follows. 
Section \ref{solving} recalls (from \cite{FS}) 
a general method of solving such equations and proves that
if there are more than $\binom{2n}{2}$ solutions, there are 
infinitely many.
Section \ref{proof} provides a general
algorithm for finding an $n$th degree equation with exactly 
$m$ solutions provided that $m\leq\binom{2n}{2}$, except when
$m$ is 4 or 16. Section \ref{cases} deals with these special cases.
The proof of theorem \ref{main} is in 
Section \ref{thm}.

\section{Solving Matrix Polynomial Equations}\label{solving}
Let $X$ satisfy
\begin{equation}\label{general}
X^n+A_{n-1}X^{n-1}+\dots+A_1X+A_0=0
\end{equation}
where $X,A_{n-1},\dots,A_0\in M_2(\C)$.
Since $X$ is a matrix over the complex
numbers, $X$ must have an eigenvalue, $\lambda$, and a
corresponding eigenvector, $v$. Then we have that
\begin{align*}
0&=(X^n+A_{n-1}X^{n-1}+\dots+A_1X+A_0)v\\
&=X^nv+A_{n-1}X^{n-1}v+\dots+A_1Xv+A_0v\\
&=\lambda^n v+A_{n-1}\lambda^{n-1}v+\dots+A_1\lambda v+A_0 v\\
&=(\lambda^nI+\lambda^{n-1}A_{n-1}+\dots+\lambda A_1+A_0)v
\end{align*}
meaning that the matrix 
$M(\lambda)=\lambda^nI+\lambda^{n-1}A_{n-1}+\dots+\lambda A_1+A_0$ must have a 
non-trivial nullspace. 
We shall call $M(t)$ the {\em corresponding polynomial matrix}. 
In order for $M(\lambda)$ to have a non-trivial nullspace, we need
\begin{equation}
\det(M(\lambda))=0.
\end{equation}
Thus we need to solve the equation $\det(M(t))=0$ to get the possible
eigenvalues. Since $\det(M(t))$ is a $2n$th degree polynomial, it 
has $2n$ roots, counting multiplicities. 
We call these roots, $\lambda_1,\dots,\lambda_{2n}$, the {\em critical values}
of the equation since they are the only possibilities for eigenvalues
of the solutions. 

If we assume that $\lambda$ is an eigenvalue of $X$, then it must have
an eigenvector, $v$. This vector can be any non-zero 
vector in the nullspace
of $M(\lambda)$. We call the nullspace of $M(\lambda)$ the 
{\em $\lambda$-critical space} and any nonzero 
vector in it a {\em $\lambda$-critical vector}.

Using this, we can now prove proposition \ref{inf}.

\begin{prop}\label{inf}
If (\ref{general}) has finitely many solutions, then there are at most 
$\binom{2n}{2}$ of them.
\end{prop}
\begin{proof}
The idea is that 
if there are $p$ distinct critical values then there are at most 
$\binom{p}{2}$ diagonalizable solutions of (\ref{general});
for each repeated root of $M(t)$ 
there is at most one non-diagonalizable solution; 
and
non-repeated roots do not have associated non-diagonalizable solutions.
Then the number of non-diagonalizable solutions is at most $2n-p$, so the total 
number of solutions is at most $\binom{p}{2}+2n-p$. Since $p\leq 2n$, we have
\begin{align*}
  \binom{p}{2} + 2n - p = 2n +\frac{p(p-3)}{2}\leq 
  2n + \frac{2n(2n-3)}{2} = \binom{2n}{2}.
\end{align*}

We will now show that there are at most $\binom{p}{2}$
diagonalizable solutions.
A diagonalizable solution is formed by taking
two distinct critical values and their corresponding critical vectors,
and using them to define a matrix. If there are $p$ 
distinct critical values, then there are at most $\binom{p}{2}$
diagonalizable solutions. 

We discount the case in which a critical value, $\lambda$, has a critical
space with dimension greater than one. This is because in this case we can
create infinitely many solutions by taking a critical value different from 
$\lambda$ and its critical space together with $\lambda$ and any of the
one-dimensional subspaces of the $\lambda$-critical space. (If there is no
critical value distinct from $\lambda$ then there is only one-diagonalizable
solution anyway.)

We will now show that if a critical value, $\lambda$, is not repeated,
then there is no non-diagonalizable solution with $\lambda$ as its eigenvalue.
This is Theorem VIII.4 in \cite{G}. If $X$ is a non-diagonalizable solution to (\ref{general}) with eigenvalue 
$\lambda$, then the characteristic polynomial of $X$ has $\lambda$
as a double root. However, $\det(M(t))$ is a multiple of the characteristic
polynomial of $X$. Explicitly:
\begin{eqnarray*}
M(t)&=&\sum_{i=1}^{n}
       \left(\sum_{j=0}^{i-1} A_{n-i+j+1}X^j\right)t^{n-i}\cdot(tI-X)\\
\det(M(t))&=&\det\left(\sum_{i=1}^{n}
             \left(\sum_{j=0}^{i-1} A_{n-i+j+1}X^j\right)t^{n-i}\right)
           \cdot\det(tI-X)
\end{eqnarray*}
where $A_n=I$.
Thus if $X$ is non-diagonalizable, then its eigenvalue must be a repeated
critical value of $M(t)$.

We now show that each repeated critical value has at most 1 non-diagonalizable 
solution associated with it.

 Let $f(X) = X^n + a_{n-1}X^{n-1} + ... + A_0.$  Assume $Y$ and $Z$ are distinct nilpotent solutions of (\ref{general}). Then, since
$Y$ and $Z$ are two by two matrices, $Y^2 = Z^2 = 0$, 
and so $0 = f(Y) = A_1Y + A_0 = f(Z) = A_1Z + A_0.$ We will show that (\ref{general}) 
has infinitely many solutions.  If $A_1 = 0$ we have $0 = A_0$.  Then any solution 
of $X^2 = 0$ is a solution of (\ref{general}) and so there are infinitely many solutions.  Now 
assume $A_1 \ne 0$.  We have $A_1Y + A_0 =  A_1Z + A_0$.  Multiplying on the right by $Z$ gives $A_1YZ + A_0Z = A_1Z^2 + A_0Z$ so that $A_1YZ = 0$.  Similarly, $A_1ZY = 0$ 
and so $A_1(YZ - ZY) = 0$.  Now $tr(YZ - ZY) = 0 $ and since $A_1 \ne 0$ we have $det(YZ - ZY) = 0$.  Thus $YZ-ZY$ is nilpotent and so  $0 = (YZ - ZY)^2 = YZYZ + ZYZY.$  But $(Y+Z)^4 = YZYZ + ZYZY$ so $Y+Z$ is nilpotent and hence $0 = (Y+Z)^2 = YZ + ZY.$  Then $\mu Y + (1- \mu)Z$ is a nilpotent solution of (\ref{general}) for any $\mu \in \mathbb{C}.$

Now if $Y$ is a non-diagonalizable $2$ by $2$ matrix, it has a single eigenvalue $\lambda$.  If $Y,Z$ are two distinct non-diagonalizable solutions of (1) with eigenvalue $\lambda$, then $Y - \lambda I, Z - \lambda I$
are distinct nilpotent solutions of $(X + \lambda I)^2 + A_1(X + \lambda I) + A_0 = 0$ so this equation and hence equation (\ref{general}) have infinitely many solutions.

Thus, if (1) has finitely many solutions, the number of non-diagonalizable solutions of (1) is bounded by the number of repeated roots, say $k$, of $M(t) = 0$.  Since the number of distinct roots is $\le 2n-k$, the number of diagonalizable solutions is $\le {{2n-k} \choose 2}$ and the total number of solutions is $\le {{2n-k} \choose 2} + k \le {{2n} \choose 2}.$

\end{proof}

\section{Finding Equations with Exactly $m$ Solutions}\label{proof}
In this section, we will show how to find equations
with exactly $m$ solutions, all of which are
diagonalizable.

The crux of being able to find our equation with
exactly $m$ solutions is choosing
the critical vectors for the equation to have. We want to
choose the vectors in such a way that exactly $m$ pairs of
these vectors are linearly independent. However, it is not
always possible to find a set of $2n$ vectors with this
property.

The first subsection shows a way to create an equation
that has only $p$ distinct critical values ($p\leq2n$) 
but does not have any 
two-dimensional critical space or non-diagonalizable 
solutions. The second subsection shows how to choose the
configuration of the $p$ critical vectors such that 
exactly $m$ pairs are linearly 
independent. The third subsection shows how to choose the 
critical values and vectors such that a polynomial with
the desired properties exists.
\subsection{Cutting down the number of critical values}
Since we may not always want $2n$ distinct critical values, 
this lemma shows a way to make an equation with only $p$ 
distinct critical values by making one of the critical
values have multiplicity $\bar{p}=2n-p+1$ while keeping
the critical space at dimension 1 and not creating any
non-diagonalizable solutions. How to choose $p$ and $\bar{p}$
is discussed in lemma \ref{vectors}.

Let $A_i = \left[\begin{array}{rr}
a_{11}^{(i)} & a_{12}^{(i)} \\
a_{21}^{(i)} & a_{22}^{(i)}\end{array}\right]$. 

\begin{lemma}\label{values} 
Let $n, \bar{p}\in\N$, $2n\geq \bar{p}$.
Assume that either
\begin{enumerate}
\item $\bar{p}\leq n$, and 
$a_{11}^{(i)}=a_{21}^{(i)}=0$ for $0\leq i\leq \bar{p}-1$\nonumber
\item $\bar{p}>n$, and 
$a_{11}^{(i)}=a_{21}^{(i)}=0$ for $0\leq i\leq n-1$ and 
$a_{22}^{(i)}=0$ for $0\leq i\leq \bar{p}-n-1$
\end{enumerate}
Then 0 is a critical value of equation (\ref{general})
with multiplicity $\geq\bar{p}$ and 
$\left[\begin{array}{r}1\\0\end{array}\right]$ 
is a 0-critical vector of (\ref{general}).

Furthermore, if condition 1 holds and 
the multiplicity of 0 is exactly $\bar{p}$, or if condition 2 holds and 
$a_{12}^{(0)} \neq 0$ then
the 0-critical space of (\ref{general}) is exactly 
$\textrm{span}\left\{\left[\begin{array}{r}1\\0\end{array}\right]\right\}$,
and there are no nilpotent solutions.
\end{lemma}
\begin{proof}\qquad
\paragraph{Case 1:} $\bar{p}\leq n$

The polynomial matrix, $M(t)$, is 
\begin{displaymath}
\left[\begin{array}{rr}
t^n + a_{11}^{(n-1)}t^{n-1} + \dots + a_{11}^{(\bar{p})}t^{\bar{p}} & 
      a_{12}^{(n-1)}t^{n-1} + \dots + a_{12}^{(0)} \\
      a_{21}^{(n-1)}t^{n-1} + \dots + a_{21}^{(\bar{p})}t^{\bar{p}} & 
t^n + a_{22}^{(n-1)}t^{n-1} + \dots + a_{22}^{(0)} 
\end{array}\right].
\end{displaymath}
The determinant of this is
\begin{displaymath}
t^{\bar{p}}\cdot det
\left[\begin{array}{rr}
t^{n-\bar{p}} + a_{11}^{(n-1)}t^{n-\bar{p}-1} + \dots + a_{11}^{(\bar{p})} & 
           a_{12}^{(n-1)}t^{n-1}     + \dots + a_{12}^{(0)} \\
           a_{21}^{(n-1)}t^{n-\bar{p}-1} + \dots + a_{21}^{(\bar{p})} & 
t^{n}    + a_{22}^{(n-1)}t^{n-1} + \dots + a_{22}^{(0)} 
\end{array}\right].
\end{displaymath}
Thus $0$ is an critical value with multiplicity at least $\bar{p}$. 

\paragraph{Case 2:} $\bar{p}>n$

The polynomial matrix, $M(t)$, is
\begin{displaymath}
\left[\begin{array}{rr}
t^n &       a_{12}^{(n-1)}t^{n-1} + \dots + a_{12}^{(0)} \\
0   & t^n + a_{22}^{(n-1)}t^{n-1} + \dots + a_{22}^{(\bar{p}-n)}t^{\bar{p}-n} 
\end{array}\right].
\end{displaymath}
As this matrix is upper triangular, its determinant is
$t^{2n}+a_{22}^{(n-1)}t^{2n-1}+\ldots + a_{22}^{(\bar{p}-n)}t^{\bar{p}}$.
Thus $0$ is an critical value with multiplicity at least $\bar{p}$.

Note that in the first case, the multiplicity of 0 is exactly $\bar{p}$
if and only if  
$\det\left(\left[\begin{array}{rr}
a_{11}^{(\bar{p})} & a_{12}^{(0)} \\
a_{21}^{(\bar{p})} & a_{22}^{(0)}
\end{array}\right]\right)\neq0$. In particular, this means that 
$\left[\begin{array}{r}a_{12}^{(0)}\\a_{22}^{(0)}\end{array}\right]\neq0$.
In the second case, we required that $a_{12}^{(0)}\neq 0$, so it is also 
true here that 
$\left[\begin{array}{r}a_{12}^{(0)}\\a_{22}^{(0)}\end{array}\right]\neq0$

To find the 0-critical space, consider 
$M(0)=\left[\begin{array}{rr} 0&a_{12}^{(0)}\\0&a_{22}^{(0)}\end{array}\right]$.
We can see that $v=\left[\begin{array}{r}1\\0\end{array}\right]
\in\textrm{null}(M(0))$, so $\left[\begin{array}{r}1\\0\end{array}\right]$ 
is a 0-critical vector. 
If the multiplicity of 0 is exactly $\bar{p}$
then $a_{12}^{(0)}$ and $a_{22}^{(0)}$ are not both $0$ so
the 0-critical space is exactly 
$\textrm{span}\{v\}$.

To show that when the multiplicity of 0 is exactly $\bar{p}$ there are no
nilpotent solutions,
assume $X$ is a nilpotent solution to (\ref{general}), 
$X=\left[\begin{array}{rr}x_{11} & x_{12} 
\\ x_{21} & x_{22}\end{array}\right]$. 
Since $X$ has eigenvalue $0$ with eigenvector 
$v=\left[\begin{array}{r}1\\0\end{array}\right]$, 
$X$ must have the form 
$\left[\begin{array}{rr}0 & x_{12} \\ 0 & x_{22}\end{array}\right]$. 
Since $X$ is non-diagonalizable, there exists a vector 
$w= \left[\begin{array}{r}w_1\\w_2\end{array}\right]$
such that 
$Xw=\left[\begin{array}{r}1\\0\end{array}\right]$.
Since $X^2=0$, we get 
\begin{eqnarray*}
0&=&X^nw+A_{n-1}X^{n-1}w+\dots+A_1Xw+A_0w\\
 &=&A_1Xw+A_0w\\
 &=& \left[\begin{array} {rr} 0 & a_{12}^{(1)} \\ 
0 & a_{22}^{(1)} \end{array}\right]Xw
+
\left[\begin{array} {rr} 0 & a_{12}^{(0)} \\ 
0 & a_{22}^{(0)} \end{array}\right]
w\\
&=&\left[\begin{array} {rr} 0 & a_{12}^{(1)} \\ 
0 & a_{22}^{(1)} \end{array}\right]v
+
\left[\begin{array} {rr} 0 & a_{12}^{(0)} \\ 
0 & a_{22}^{(0)} \end{array}\right]
w\\
&=&\left[\begin{array}{r}0\\0\end{array}\right]
+
\left[\begin{array} {r} a_{12}^{(0)}w_2 \\ a_{22}^{(0)}w_2 \end{array}\right].
\end{eqnarray*}
Since $a_{12}^{(0)}$ and $a_{22}^{(0)}$ are not both $0$, $w_2=0$, 
so $w=\left[\begin{array}{r}w_1\\0\end{array}\right]\in\textrm{span}\{v\}$.
Thus 
$\left[\begin{array}{r}1\\0\end{array}\right]
=Xw=\left[\begin{array}{r}0\\0\end{array}\right]$, a contradiction.
\end{proof}

\subsection{Cutting down the number of linearly independent pairs of critical vectors}
This lemma shows how many of the 
critical values should be distinct, and how 
to construct a configuration of vectors
such that exactly $m$ pairs are linearly independent. 

Recall that the critical value $0$ plays a distinguished role as
it may have multiplicity $>1$.

\begin{lemma}\label{vectors} 
For any $m\in\N$, $m\neq 4,16$, let $p$ be the integer such that 
\begin{equation}\label{p}
{p-1 \choose 2}<m\leq{p \choose 2}.
\end{equation}
There exists an equivalence relation $\sim$ on $\mathbb{N}_p$, where 
$\mathbb{N}_p=\{0,1,2,\dots,p-1\}$,
such that 
\begin{enumerate}
\item the set 
$\{(x,y)\in\mathbb{N}_p\times\mathbb{N}_p | x<y\text{, and } x\not \sim y\}$
has exactly $m$ elements
\item $0$ is in an equivalence class by itself.
\item No equivalence class has more than $\lceil\frac{p}{2}\rceil$ elements
\end{enumerate}
\end{lemma}
\begin{proof}
We will define $\sim$ to satisfy 1 and 2 in each of several cases.
We will defer the verification that $\sim$ also satisfies 3 until 
the end of the proof.

When we define $\sim$, we want the number of pairs 
that are not equivalent to be $m$,
so we want the number of pairs that are equivalent to be
 ${p \choose 2}-m$.
Let $a$ and $b$ be the integers such 
that ${p \choose 2}-m=3a+b$ where $0\leq b<2$.

Now, by (\ref{p}), 
\begin{eqnarray}
{p-1 \choose 2} &<& m\nonumber \\
1-p={p-1 \choose 2}-{p \choose 2} &<& m-{p \choose 2}\nonumber\\
3a+b+1={p \choose 2}-m+1 &<& p\label{pineq}
\end{eqnarray}
so $3a+b+1\in\N_p$.

\paragraph{Case 1:} $b=0$\\
Let $\sim$ be the equivalence relation corresponding 
to the following partition of $\N_p$
{\setlength\arraycolsep{0pt}
\begin{eqnarray*}
\{&\{&0\},\\
  &\{&1,2,3\},\{4,5,6\},\dots,\{3a-2, 3a-1, 3a\},\\
  &\{&3a+1\},\{3a+2\},\dots,\{p-1\}\}.
\end{eqnarray*}}
Note that since $3a+b+1\in\N_p$, this is a partition of $\N_p$.
Since each set of three creates 3 equivalent pairs, and there are $a$
groups of three, there are $3a={p \choose 2}-m$ pairs that are
equivalent and $m$ pairs are not.

\paragraph{Case 2:} $b=1$\\
Let $\sim$ be the equivalence relation corresponding 
to the following partition of $\N_p${\setlength\arraycolsep{0pt}
\begin{eqnarray*}
\{&\{&0\},\\
  &\{&1,2,3\},\{4,5,6\},\dots,\{3a-2, 3a-1, 3a\},\\
  &\{&3a+1,3a+2\}\\
  &\{&3a+3\},\{3a+4\},\dots,\{p-1\}\}.
\end{eqnarray*}}
Note that since $3a+b+1\in\N_p$, this is a partition of $\N_p$.
Each set of three creates 3 equivalent pairs, and the set of two 
creates 1 equivalent pair. There are $a$
sets of three, and 1 set of two so there are $3a+1={p \choose 2}-m$ 
pairs that are equivalent and $m$ pairs are not.

\paragraph{Case 3a:} $b=2$ and $a \geq 2$\\
Let $\sim$ be the equivalence relation corresponding 
to the following partition of $\N_p$
{\setlength\arraycolsep{0pt}
\begin{eqnarray*}
\{&\{&0\},\\
  &\{&1,2,3\},\{4,5,6\},\dots,\{3a-8, 3a-7, 3a-6\},\\
  &\{&3a-5, 3a-4, 3a-3, 3a-2\},\\
  &\{&3a-1,3a\},\{3a+1,3a+2\},\\
  &\{&3a+3\},\{3a+4\},\dots,\{p-1\}\}.
\end{eqnarray*}}

Each set of three creates 3 equivalent pairs, 
each set of two creates 1 equivalent pair, 
and the set of four creates 6 equivalent pairs.
Thus there are $3(a-2)+6+2=3a+2={p \choose 2}-m$ pairs that are equivalent 
and $m$ pairs that are not.

\paragraph{Case 3b:} $b=2$ and $p>3a+4$
Let $\sim$ be the equivalence relation corresponding 
to the following partition of $\N_p$
{\setlength\arraycolsep{0pt}
\begin{eqnarray*}
\{&\{&0\},\\
  &\{&1,2,3\},\{4,5,6\},\dots,\{3a-2, 3a-1,3a\},\\
  &\{&3a+1,3a+2\},\{3a+3,3a+4\},\\
  &\{&3a+5\},\{3a+6\},\dots,\{p-1\}\}.
\end{eqnarray*}}

Each set of three creates 3 equivalent pairs,
each set of two creates 1 equivalent pair,
Thus there are $3a+2={p \choose 2}-m$ pairs that are equivalent and $m$ pairs
that are not.

Note that the two parts of the $b=2$ case are not 
mutually exclusive, and they do
not cover all the $b=2$ cases. The cases that are not covered are the ones
in which $b=2$, $a<2$, and $p\leq 3a+4$.

If $b=2$, $a=0$ and $p\leq 3a+4$,then $p\leq 4$. But by (\ref{pineq}),
$3=3a+b+1<p$ so $p=4$ and hence $m=\binom{4}{2}-2=4$. Thus $m=4$ is
one of our special cases.

If $b=2$, $a=1$, and $p\leq 3a+4$, then $p\leq 7$. But by (\ref{pineq}),
$6=3a+b+1<p$, so $p=7$ and hence $m=\binom{7}{2}-5=16$
So $m=16$ is our other special case.

Our construction of the equivalence relations satisfied the first 
two conditions. We have yet to prove that no equivalence class
has more than $\lceil\frac{p}{2}\rceil$ elements.

First of all, notice that none of our constructed equivalence classes
have more than $4$ elements, so if $p\geq 7$, then we are done.

The only case in which we have more than $3$ elements in one equivalence
class is when $b=2$ and $a\geq2$. But then, by (\ref{pineq}), 
\begin{equation*}
9\leq 3a+b+1 < p.
\end{equation*}
Thus case 3a does not occur when $p<7$, so when $p\leq 9$,
 we will not have any equivalence classes with more than 
3 elements. Thus if $p\geq 5$, we are done.

If $p\leq4$, then by (\ref{pineq}), $3a+b+1<p \leq 4$ so $3a+b\leq2$
and $a=0$. However, if $a=0$ then none of the other cases require
an equivalence class of three elements. Thus if $p\leq4$, all equivalence
classes have at most two elements, so for $p\geq3$, we are done.

If $p\leq2$, then $m=\binom{p}{2}$, so $a=b=0$. This falls into case
1 with none of the elements being equivalent to anything other than 
itself. 

Now we have satisfied all three conditions, so except 
for our special cases ($m=16$ and $m=4$), we can create the 
required equivalence relation.
\end{proof}

\subsection{Choosing the numbers}
The previous lemma ensured that no equivalence class has more
than $\lceil\frac{p}{2}\rceil$ elements. Since $n,p\in\N$ and
$p \leq 2n$, this means that no equivalence class has more 
than $n$ elements.

\begin{lemma}\label{entries}
Given
\begin{enumerate}
\item $n\in\N$
\item $p\in\N$, $2 \leq p \leq 2n$
\item an equivalence relation, $\sim$ on $\N_p$, 
such that 0
is only equivalent to itself, and no equivalence class has more
than $n$ elements
\end{enumerate}
it is possible to choose 2-dimensional vectors 
$v_0, v_1, \dots, v_{p-1}$ and scalars
$\lambda_0, \lambda_1, \dots, \lambda_{p-1}$ such that 
\begin{enumerate}
\item $v_i$ and $v_j$ are linearly dependent if and only if $i\sim j$
\item there exists an $n$th degree 
equation over 2 by 2 matrices that has critical values
$\lambda_0, \lambda_1, \dots, \lambda_{p-1}$, 
and the critical space corresponding
to $\lambda_i$ is $span(v_i)$.
\item the equation has no non-diagonalizable solutions.
\end{enumerate}
\end{lemma}

\begin{proof}\quad
\paragraph{Choosing the critical vectors}
Let
$v_0 = \left[\begin{array}{r} 1\\0 \end{array}\right]$
and let all the others be of the form
$v_i = \left[\begin{array}{r} 1\\y_i \end{array}\right]$
where $y_i\neq0$ and $v_i=v_j\Leftrightarrow i\sim j$.

\paragraph{Choosing the critical values}
Let $\lambda_0=0$ and
let $\lambda_i$ be an $n$th root of $y_i$.
Furthermore, take the $\lambda_i$ to be distinct.
This is possible because there are at most $n$ numbers in the same
equivalence class, so no more than $n$ of the $y_i$ will be the same. 

\paragraph{Tying it together and proving that it works}
Let $\bar{p}=2n-p+1$. If $\bar{p}\leq n$, we use case 1 of
lemma \ref{values} and our equation will have a corresponding
polynomial matrix of the form
\begin{equation*}
\left[\begin{array}{rr}
t^n + a_{11}^{(n-1)}t^{n-1} + \dots + a_{11}^{(\bar{p})}t^{\bar{p}} & 
      a_{12}^{(n-1)}t^{n-1} + \dots + a_{12}^{(0)} \\
      a_{21}^{(n-1)}t^{n-1} + \dots + a_{21}^{(\bar{p})}t^{\bar{p}} & 
t^n + a_{22}^{(n-1)}t^{n-1} + \dots + a_{22}^{(0)} 
\end{array}\right].
\end{equation*}
We need to choose the $a_{ij}^{(k)}$
so that when we plug in $\lambda_i$ for $t$, we 
get a matrix with a null space spanned by $v_i$.
First let us focus on finding the $a_{1j}^{(k)}$, i.e., on finding the 
first row of each $A_k$. (Finding the $a_{2j}^{(k)}$ is similar.)
Since the first component of 
$M(\lambda_i)v_i$ must be 0, we obtain the equation
\begin{equation*}
a_{11}^{(\bar{p})}\lambda_i^{\bar{p}} + \dots +
a_{11}^{(n-1)}\lambda_i^{n-1} +  
a_{12}^{(0)}y_i +\dots +
a_{12}^{(n-1)}\lambda_i^{n-1}y_i=-\lambda_i^n.
\end{equation*}
The corresponding system of $p-1$ linear equations (in the $a_{ij}^{(k)}$)
may be written in the following matrix form

\begin{displaymath}
\left[\begin{array}{cccccccc}
\lambda_1^{\bar{p}} & \lambda_1^{\bar{p}+1} & \dots 
       & \lambda_1^{n-1} 
& y_1 & \lambda_1 y_1 & \dots 
       & \lambda_1^{n-1}y_1 \\
\lambda_2^{\bar{p}} & \lambda_2^{\bar{p}+1} & \dots 
       & \lambda_2^{n-1} 
& y_2 & \lambda_2 y_2 & \dots 
       & \lambda_2^{n-1}y_2 \\
&&&&\vdots&&&\\
\lambda_{p-1}^{\bar{p}} & \lambda_{p-1}^{\bar{p}+1} & \dots 
       & \lambda_{p-1}^{n-1} 
& y_{p-1} & \lambda_{p-1} y_{p-1} & \dots 
       & \lambda_{p-1}^{n-1}y_{p-1} \\
\end{array}\right]
\left[\begin{array}{r}
a_{11}^{(\bar{p})} \\ a_{11}^{(\bar{p}+1)} \\ \vdots \\a_{11}^{(n-1)} \\
a_{12}^{(0)} \\ a_{12}^{(1)} \\ \vdots \\a_{12}^{(n-1)}
\end{array}\right]
=
\left[\begin{array}{r}
-\lambda_1^n \\ -\lambda_2^n \\ \vdots \\ -\lambda_{p-1}^n
\end{array}\right].
\end{displaymath}

(We do not need to include the $0$th equation 
because lemma \ref{values} ensures that $\lambda_0$ and $v_0$ will
be a critical value/critical vector pair.)
Because we chose each $\lambda_i$ to be an $n$th root of $y_i$, 
we have that $y_i=\lambda_i^n$. Thus this equation can be written

\begin{equation*}
\left[\begin{array}{rrrr}
\lambda_1^{\bar{p}} & \lambda_1^{\bar{p}+1} & \dots & \lambda_1^{2n-1}  \\
\lambda_2^{\bar{p}} & \lambda_2^{\bar{p}+1} & \dots & \lambda_2^{2n-1}  \\
&\vdots&&\\
\lambda_{p-1}^{\bar{p}}  & \lambda_{p-1}^{\bar{p}+1} 
& \dots & \lambda_{p-1}^{2n-1} 
\end{array}\right]
\left[\begin{array}{r}
a_{11}^{(\bar{p})} \\ a_{11}^{(\bar{p}+1)} \\ \vdots \\a_{11}^{(n-1)} \\
a_{12}^{(0)} \\ a_{12}^{(1)} \\ \vdots \\a_{12}^{(n-1)}
\end{array}\right]
=
\left[\begin{array}{r}
-\lambda_0^n \\ -\lambda_1^n \\ -\lambda_2^n \\ \vdots \\ -\lambda_{p-1}^n
\end{array}\right].
\end{equation*} 
Note that the coefficient matrix is square ($2n-\bar{p}=p-1$).
The determinant of this matrix is
\begin{eqnarray*}
&&\det\left(\left[\begin{array}{rrrr}
\lambda_1^{\bar{p}} & \lambda_1^{\bar{p}+1} & \dots & \lambda_1^{2n-1}  \\
\lambda_2^{\bar{p}} & \lambda_2^{\bar{p}+1} & \dots & \lambda_2^{2n-1}  \\
&\vdots&&\\
\lambda_{p-1}^{\bar{p}}  & \lambda_{p-1}^{\bar{p}+1} 
& \dots & \lambda_{p-1}^{2n-1} 
\end{array}\right]\right)\\
&=&
\prod_{i=1}^{p-1} \lambda_i^{\bar{p}}
\cdot
\det\left(\left[\begin{array}{rrrrr}
1 & \lambda_1 & \lambda_1^{2} & \dots & \lambda_1^{p-2}  \\
1 & \lambda_2 & \lambda_2^{2} & \dots & \lambda_2^{p-2}  \\
&\vdots&&\\
1 & \lambda_{p-1} & \lambda_{p-1}^{2} 
& \dots & \lambda_{p-1}^{p-2} 
\end{array}\right]\right)\\
&=&
\prod_{i=1}^{p-1} \lambda_i^{\bar{p}}
\cdot
\prod_{i=1}^{p-1} \prod_{j=i+1}^{p-1} (\lambda_j-\lambda_i)
\end{eqnarray*}
(with the last equality following from 
the formula for the determinant of a 
Vandermonde matrix). Since none of 
$\lambda_1,\dots,\lambda_{p-1}$ are zero, and all of our $\lambda_i$ are 
distinct, the determinant is non-zero 
so the system of equations is consistent.

If $\bar{p}>n$, then take $a_{11}^{(i)}=a_{21}^{(i)}=0$ 
for $0\leq i\leq n-1$, 
$a_{22}^{(i)}=0$ for $0\leq i\leq \bar{p}-n-1$, 
as specified by lemma \ref{values}. 
Also take $a_{12}^{(i)}=0$ for $p-1\leq i\leq n-1$. 
This does not conflict with the restriction that $a_{12}^{(0)}\neq 0$
because we required $2\leq p$.
The matrix polynomial looks like
\begin{equation*}
\left[\begin{array}{rr}
t^n &       a_{12}^{(p-2)}t^{p-2} + \dots + a_{12}^{(0)} \\
0   & t^n + a_{22}^{(n-1)}t^{n-1} + \dots + a_{22}^{(\bar{p}-n)}t^{\bar{p}-n} 
\end{array}\right].
\end{equation*}
To find the coefficients of the first row, we set up the equations
in the matrix form
\begin{equation*}
\left[\begin{array}{rrrrr}
y_1 & y_1 \lambda_1 & y_1\lambda_1^2 & \dots & \lambda_1^{p-2}y_1\\
y_2 & y_2 \lambda_2 & y_2\lambda_2^2 & \dots & \lambda_2^{p-2}y_2\\
&&\vdots&&\\
y_{p-1} & y_{p-1} \lambda_{p-1} 
& y_{p-1}\lambda_{p-1}^2 & \dots & \lambda_{p-1}^{p-2}y_{p-1}
\end{array}\right]
\left[\begin{array}{r}
a_{12}^{(0)} \\ a_{12}^{(1)} \\ \vdots \\ a_{12}^{(p-2)}
\end{array}\right]
=
\left[\begin{array}{r}
-\lambda_1^n \\ -\lambda_2^n \\ \vdots \\ -\lambda_{p-1}^n
\end{array}\right].
\end{equation*}

Since $y_i = \lambda_i^n$ for $1\leq i\leq p-1$, we see that the first 
column of the coefficient matrix is the negative of the vector on the 
right-hand side of the equation. Thus
\begin{equation*}
  \left[\begin{array}{c}
      a_{12}^{(0)} \\a_{12}^{(1)}\\ \vdots \\a_{12}^{(p-2)}
  \end{array}\right]
  =
  \left[\begin{array}{c}
      -1\\0\\\vdots\\0
  \end{array}\right]
\end{equation*}
is a solution to the system of equations. Note that this shows
 $a_{12}^{(0)}\neq 0$ so the hypothesis of Lemma \ref{values} is satisfied.

To find the coefficients of the second row, we use the matrix equation
\begin{equation*}
\left[\begin{array}{rrr}
\lambda_1^{\bar{p}-n}y_1 & \dots & \lambda_1^{n-1}y_1  \\
\lambda_2^{\bar{p}-n}y_2 & \dots & \lambda_2^{n-1}y_2  \\
&\vdots&\\
\lambda_{p-1}^{\bar{p}-n}y_{p-1} & \dots & \lambda_{p-1}^{n-1} y_{p-1}
\end{array}\right]
\left[\begin{array}{r}
a_{22}^{(\bar{p}-n)} \\ \vdots \\a_{22}^{(n-1)}
\end{array}\right]
=
\left[\begin{array}{r}
-\lambda_1^ny_1 \\ -\lambda_2^ny_2 \\ \vdots \\ -\lambda_{p-1}^n y_{p-1}
\end{array}\right]
\end{equation*}

and the determinant of this matrix is

\begin{eqnarray*}
&&\det\left(\left[\begin{array}{rrr}
\lambda_1^{\bar{p}-n}y_1 & \dots & \lambda_1^{n-1}y_1  \\
\lambda_2^{\bar{p}-n}y_2 & \dots & \lambda_2^{n-1}y_2  \\
&\vdots&\\
\lambda_{p-1}^{\bar{p}-n} y_{p-1}& \dots & \lambda_{p-1}^{n-1} y_{p-1}
\end{array}\right]\right)\\
&=&\det\left(\left[\begin{array}{rrrr}
\lambda_1^{\bar{p}} & \lambda_1^{\bar{p}+1} & \dots & \lambda_1^{2n-1}  \\
\lambda_2^{\bar{p}} & \lambda_2^{\bar{p}+1} & \dots & \lambda_2^{2n-1}  \\
&\vdots&&\\
\lambda_{p-1}^{\bar{p}}  & \lambda_{p-1}^{\bar{p}+1} 
& \dots & \lambda_{p-1}^{2n-1} 
\end{array}\right]\right)\\
&=&
\prod_{i=1}^{p-1} \lambda_i^{\bar{p}}
\cdot
\det\left(\left[\begin{array}{rrrrr}
1 & \lambda_1 & \lambda_1^{2} & \dots & \lambda_1^{p-2}  \\
1 & \lambda_2 & \lambda_2^{2} & \dots & \lambda_2^{p-2}  \\
&\vdots&&\\
1 & \lambda_{p-1} & \lambda_{p-1}^{2} 
& \dots & \lambda_{p-1}^{p-2} 
\end{array}\right]\right)\\
&=&
\prod_{i=1}^{p-1} \lambda_i^{\bar{p}}
\cdot
\prod_{i=1}^{p-1} \prod_{j=i+1}^{p-1} (\lambda_i-\lambda_j).
\end{eqnarray*}

Since $\lambda_1,\dots\lambda_{p-1}$ are all distinct and 
non-zero, the determinant is non-zero so this system is solvable.

Because we chose $p-1$ non-zero distinct critical values for the
equation, 0 cannot have a multiplicity of greater 
than $2n-(p-1)=\bar{p}$. Lemma \ref{values} assures us that 0
will have multiplicity at least $\bar{p}$, so 0 must have 
multiplicity exactly $\bar{p}$. Then we can use the second
part of lemma \ref{values}. Thus the 0-critical space is 
1-dimensional, and (\ref{general}) has no nilpotent solutions.
\end{proof}
\section{Special Cases}\label{cases}
For \textbf{$m=4$}, we must have $4=m\leq\binom{2n}{2}$, 
so $n\geq 2$. Then we may use the equation whose corresponding 
polynomial matrix is 
\begin{equation*}
\left[\begin{array}{rr}
(t-1)(t+1)^{n-1} & 0\\
0 & (t-2)(t+2)^{n-1}
\end{array}\right].
\end{equation*}
To show that this has no non-diagonalizable solution with 
an eigenvalue $-1$ in general, use a change 
of variables $u=t+1$ so we are looking for a non-diagonalizable
solution with eigenvalue 0.  
The polynomial matrix becomes
\begin{equation*}
\left[\begin{array}{rr}
(u-2)u^{n-1} & 0\\
0 & (u-3)(u+1)^{n-1}
\end{array}\right]
\end{equation*}
the 0-critical space of this equation is
$\textrm{span}\left\{\left[\begin{array}{r}1\\0\end{array}\right]\right\}$
so $X$ must have the form 
$\left[\begin{array}{rr}0 & x_{12} \\ 0 & x_{22}\end{array}\right]$.
Since $X$ is nilpotent, $x_{22}=0$. 
Then we have that 
\begin{eqnarray*}
0&=&A_1 X +A_0\\
&=& \left[\begin{array}{rr}a_{11}^{(1)} & 0 \\ 0 & a_{22}^{(1)}\end{array}\right]
 \left[\begin{array}{rr}0 & x_{12} \\ 0 & 0\end{array}\right]+
 \left[\begin{array}{rr}0 & 0 \\ 0 & -3\end{array}\right]\\
&=&\left[\begin{array}{rr}0 & a_{11}^{(1)} x_{12} \\ 0 & -3\end{array}\right].
\end{eqnarray*}
This is a contradiction, thus there are no non-diagonalizable solutions.
The argument for why there is no non-diagonalizable solutions
with $-2$ as an eigenvalue is similar.

For \textbf{$m=16$}, we must have $16=m\leq\binom{2n}{2}$, so $n\geq4$.
Then we may use the equation whose corresponding
matrix polynomial is 
\begin{equation*}
\left[\begin{array}{rr}
(t-3)(t+3)(t-1)(t+1)^{n-3} & 0\\
0 & (t-4)(t+4)(t-2)(t+2)^{n-3}
\end{array}\right].
\end{equation*}
It is straightforward why there are exactly 16 diagonalizable 
solutions,
and the reason why there are no non-diagonalizable solutions
is the same as for when $m=4$.
\section{Conclusion}\label{thm}
This section proves the main result of this paper, theorem \ref{main}. 
We assume $m,n\in\N$, $m\leq \binom{2n}{2}$, 
and show that there exists an $n$th degree equation
over $2\times2$ complex matrices that has exactly $m$ solutions.

If $m \neq 4, 16$, use lemma (\ref{vectors}) and $m$ to 
define the number of distinct
critical values that we want our equation to have, $p$, and the 
equivalence relation, $\sim$, that we want to determine which 
critical spaces will be distinct. 

All the conditions of lemma (\ref{entries}) are satisfied 
by the conclusions of lemma (\ref{vectors}) except $p\leq 2n$. 
This condition is satisfied because
$\binom{p-1}{2} < m \leq \binom{2n}{2}$ so 
$p-1 < 2n$ so we get $p\leq 2n$ as desired. 
We can then use lemma (\ref{entries}) to produce our equation.

Since lemma (\ref{entries}) promises that we will not have any
non-diagonalizable solutions and that all our critical spaces
will be one-dimensional, all our solutions will have two
different eigenvalues. By our choice of $\sim$, we have exactly
$m$ ways to choose two critical value/critical vector pairs such
that the critical vectors are linearly independent. Thus we have exactly
$m$ solutions. 
The special cases of $m=4,16$ have been dealt with in section 
\ref{cases} so the proof is complete.

Address: \\
Rutgers University, Department of Mathematics \\
110 Frelinghuysen  Rd \\
Piscataway, NJ 08854\\
E-mail address: mslusky@gmail.com

\begin{thebibliography}{GLR}
\bibitem[FS]{FS} \textsc{D. Fuchs, A. Schwarz.}, A Matrix Vieta Theorem, 
E.B. Dynkin Seminar, Amer. Math. Soc. Transl. Ser 2 \textbf{196} (1996).
\bibitem[BW]{BW} \textsc{U. Bray, G. Whaples}, Polynomials with Coefficients
from a Division Ring,  Canad. J. Math, \textbf{35} (1983), 509-515
\bibitem[G]{G} \textsc{F. Gantmakher}, The theory of matrices, Vol. 1, 
Chelsea, New York 1959
\bibitem[GLR]{GLR} \textsc{I. Gohberg, P. Lancaster, and L. Rodman}, 
\textit{Matrix polynomials}, Academic Press, 1982
\bibitem[LR]{LR} \textsc{P. Lancaster and L. Rodman}, 
\textit{Algebraic Riccati Equations},
Clarendon Press, 1995
\end{thebibliography}
\end{document}